\documentclass[a4paper,10pt]{article}
\usepackage{mathtext}
\usepackage[T1,T2A]{fontenc}
\usepackage[cp1251]{inputenc}
\usepackage[english]{babel}
\usepackage{amsmath}
\usepackage{amsfonts}
\usepackage{amssymb}
\usepackage{mathrsfs}
\usepackage{amsthm}
\usepackage{enumerate}
\usepackage{graphicx}

\usepackage{color}
\usepackage{euscript}

\usepackage{cite}

\newtheorem{Le}{Lemma}[section]
\newtheorem{Def}{Definition}[section]
\newtheorem{St}[Le]{Proposition}
\newtheorem{Th}{Theorem}[section]

\newtheorem{Rem}[Le]{Remark}

\numberwithin{equation}{section}

\newcommand{\R}{\mathbb{R}}
\newcommand{\Co}{\mathbb{C}}
\newcommand{\N}{\mathbb{N}}

\newcommand{\eps}{\varepsilon}

\newcommand{\eq}[1]{\begin{equation}{#1}\end{equation}}
\newcommand{\mlt}[1]{\begin{multline}{#1}\end{multline}}
\newcommand{\alg}[1]{\begin{align}{#1}\end{align}}

\newcommand{\set}[2]{\{{#1}\mid{#2}\}}
\newcommand{\Set}[2]{\Big\{{#1}\,\Big|\;{#2}\Big\}}
\newcommand{\scalprod}[2]{\langle{#1},{#2}\rangle}
\newcommand{\fdot}{\,\cdot\,}
\newcommand{\Leqref}[1]{\stackrel{\scriptscriptstyle{\eqref{#1}}}{\leq}}

\newcommand{\Lref}[1]{\stackrel{#1}{\leq}}

\newcommand{\Eeqref}[1]{\stackrel{\scriptscriptstyle\eqref{#1}}{=}}

\newcommand{\Gref}[1]{\stackrel{#1}{\geq}}

\DeclareMathOperator{\BV}{BV}

\DeclareMathOperator{\supp}{supp}
\DeclareMathOperator{\spec}{spec}
\DeclareMathOperator{\Heat}{H}

\newcommand{\Me}{\boldsymbol{\mathrm{M}}}

\newcommand{\Sw}{\mathcal{S}}
\newcommand{\W}{\boldsymbol{\mathrm{W}}}

\newcommand{\dH}{\dim_{\mathrm{H}}}
\newcommand{\ldH}{\underline{\dim}_{\mathrm{H}}}
\DeclareMathOperator{\Tan}{Tan}
\newcommand{\Ext}{\mathcal{E}}

\title{Dimension estimates for vectorial measures with restricted spectrum}
\author{Dmitriy Stolyarov\thanks{Supported by the Russian Science Foundation grant N 19-71-30002.}}
\begin{document}
\maketitle
\begin{abstract}
We link the problem of estimating the lower Hausdorff dimension of PDE or Fourier constrained measures with Harnack's inequalities for the heat equation. Our approach provides new estimates in the case of Fourier constraints.
\end{abstract}

\section{Introduction}\label{S1}
Let~$l$ and~$d$ be natural numbers. The space~$\Me(\R^d,\Co^l)$ consists of all charges ($\Co^l$-valued sigma-additive Borel set functions) of finite variation. Here and in what follows we distinguish measures that are always scalar and non-negative from charges, which may be either~$\R$, or~$\Co$, or~$\Co^l$ valued. We define the norm in that space by the formula
\eq{
\|\mu\|_{\Me(\R^d,\Co^l)} = \sup \Set{\int\limits_{\R^d} f\,d\mu}{\|f\|_{C_0(\R^d,\Co^l)} \leq 1}.
}
Here~$C_0(\R^d,\Co^l)$ is the space of continuous functions that tend to zero at infinity, equipped with the standard sup-norm.  Alternatively, one may use the classical definition of the total variation via partitions.

Let~$k \leq l$ be natural numbers. Let~$\Omega\colon S^{d-1}\to G(l,k)$ be a smooth mapping. The notation~$S^{d-1}$ and~$G(l,k)$ means the unit sphere in~$\R^d$ and the (complex) Grassmannian, i.e. the set of all (complex) linear~$k$-dimensional subspaces of~$\Co^l$. The map~$\Omega$ gives rise to a generalization of the~$\BV$-space
\eq{\label{DefOfBV}
\BV^\Omega = \Set{\mu\in \Me(\R^d,\Co^l)}{\forall \xi \in \R^d \setminus \{0\} \quad \hat{\mu}(\xi) \in \Omega\big(\frac{\xi}{|\xi|}\big)}.
}
This space inherits the norm from the space~$\Me$. Here and in what follows we use the standard Harmonic Analysis normalization of the Fourier transform
\eq{
\hat{f}(\xi) =\mathcal{F}[f](\xi) = \int\limits_{\R^d} f(x)e^{-2\pi i\scalprod{\xi}{x}}\,dx;\quad \hat{\mu}(\xi) = \int\limits_{\R^d}e^{-2\pi i \scalprod{\xi}{x}}d\mu(x).
} 

The basic example is~$\Omega(\zeta) = \Co\zeta$, i.e.~$l = d$,~$k=1$, and~$\Omega$ maps a point~$\zeta\in S^{d-1}$ to the line it spans. In this case,~$\BV^\Omega$ becomes the classical~$\BV$ space (more precisely, it is the space of gradients of~$\BV$ functions). Other examples may be found in~\cite{Stolyarov2020HLS}.

Since we are working with the Fourier transform, we will need the Schwartz class. We denote it by~$\Sw(\mathbb{R}^d)$ or~$\Sw(\R^d,\Co^l)$ depending on whether we consider scalar or vector valued functions. By the spectrum of a function or of a ditribution we mean the support of its Fourier transform.

\begin{Rem}
The space~$\BV^\Omega$ is closed in~$\Me(\R^d,\Co^l)$. This space is also translation and dilation invariant. 
\end{Rem}
Measures in~$\BV^\Omega$ cannot be excessively singular. We will be measuring singularity on a rough scale. Recall the notion of the lower Hausdorff dimension of a charge:
\eq{
\ldH \mu = \set{\inf \alpha}{\exists \ \hbox{a Borel set}\ F \hbox{ such that } \dH F \leq \alpha, \ \mu(F) \ne 0}.
}
The question going back to~\cite{RoginskayaWojciechowski2006} is: what is the smallest possible lower Hausdorff dimension of~$\mu \in \BV^\Omega\setminus \{0\}$? The sets~$\Omega^{-1}(a)$, where~$a \in S^{d-1}$, play an important role in the subject:
\eq{
\Omega^{-1}(a) = \set{\zeta \in S^{d-1}}{a\in \Omega(\zeta)}.
}
Note that these sets are closed. We are ready to formulate our main result.
\begin{Th}\label{DimensionalEstimate}
Let~$k \in [0..d]$ be such that for any~$a\in S^{d-1}$ there exists a~$k$-dimensional linear plane~$L_a$ in~$\R^d$ such that
\eq{\label{Antisymmetry}
\Omega^{-1}(a)\cap (-\Omega^{-1}(a))\cap L_a = \varnothing.
} 
Then\textup,~$\ldH \mu \geq k$ for any~$\mu \in \BV^\Omega \setminus \{0\}$.
\end{Th}
The case~$k=d$ in Theorem~\ref{DimensionalEstimate} follows from the celebrated Uchiyama theorem from~\cite{Uchiyama1982}; the latter theorem is a generalization of the classical Fefferman--Stein theorem. Namely, the Uchiyama theorem says that if
\eq{\label{ClassicalAntisymmetry}
\forall \zeta \in S^{d-1}\qquad \Omega(\zeta)\cap \Omega(-\zeta) = \{0\},
} 
then~$\BV^\Omega \subset \mathcal{H}_1(\R^d,\Co^l)$, where~$\mathcal{H}_1$ stands for the real Hardy class. The functions~$\Omega$ satisfying~\eqref{ClassicalAntisymmetry} (i.e.~\eqref{Antisymmetry} with~$k=d$) are called antisymmetric; this condition appeared in~\cite{Janson1977} for the first time. The case~$k=1$ in Theorem~\ref{DimensionalEstimate} was settled in~\cite{RoginskayaWojciechowski2006} under much weaker regularity assumptions on~$\Omega$. The case where~$\Omega$ is a rational function was considered in~\cite{APHF2019} (in the case where~$\Omega$ is a symmetric function,~\eqref{Antisymmetry} is equivalent to the condition of~\cite{APHF2019} that the~$k$-wave cone
\eq{
\Lambda_k^\Omega = \bigcap_{L\in G(d,k)}\bigcup_{\zeta \in S^{d-1}\cap L}\Omega(\zeta)
}
is empty; we are using the real Grassmannian in this formula). In fact, in~\cite{APHF2019} it was proved that~$\mu$ is absolutely continuous with respect to the corresponding Hausdorff measure~$\mathcal{H}^{k}$ alongside with more delicate results on rectifiability. See the paper~\cite{Arroyo-Rabasa2020} for an elementary approach to some of these questions. The paper~\cite{AyoushWojciechowski2017} contains partial results in the case of arbitrary~$\Omega$ and~$k=2$ (the assumptions in that paper are stronger than~\eqref{Antisymmetry}). The paper~\cite{ASW2018} suggests a version of the dimension problem in the setting of discrete time uniform martingales; the Fourier constraints we are studying here are replaced with linear constraints on martingale differences. In that setting,~\cite{ASW2018} contains a complete solution to the dimension problem.

Our approach seemingly differs from the preceding ones. It consists of two main steps. To formulate the first step, we need the notion of the local lower Hausdorff dimension of a charge:
\eq{
\ldH \mu (x) = \varliminf_{r\to 0} \frac{\log |\mu|(B_r(x))}{\log r}, \quad x\in \supp\mu,
}
here~$B_r(x)$ stands for the closed Euclidean ball of radius~$r$ centered at~$x$. By a tempered measure or tempered charge we mean a locally finite measure or locally finite charge~$\mu$ for which there exists~$N\in\N$ such that
\eq{
|\mu|(B_R(0)) \lesssim R^N,\quad R \geq 1.
}
In other words, a tempered measure is a measure that is a Schwartz distribution. The notation~$A\lesssim B$ means~$A\leq CB$ and the constant~$C$ is uniform with respect to a certain parameter that is either clear from the context or specified somewhere nearby. For example, in the inequality above the constant~$C$ should be independent of~$R$ (however, as usual, it may depend on~$N$). 

We will be also using the standard notation for dilation of measures and tangent cones, see~\cite{Preiss1987}. By the classical convergence of locally finite measures and charges in~$\R^d$ we mean weak-* convergence on compact sets (as in~\cite{Preiss1987}). There is a natural tensor product structure on the space~$\Me(\R^d,\Co^l)$; the charge~$a\otimes \nu$, where~$\nu$ is a scalar charge and~$a\in \Co^l$, is defined by the formula
\eq{
a\otimes \nu(B) = \nu(B) a,\quad B\ \hbox{is a Borel set}.
}
\begin{Th}\label{TangentMeasureTheorem}\footnote{A theorem with the same assumptions and a stronger conclusion~$\ldH \nu \leq \ldH \mu + \eps$ was proved by Fedor Nazarov some time ago. Unfortunately, the proof had not been written down and got lost! I proved the weaker statement presented here independently and believe that my proof is different. Rami Ayoush and Michal Wojciechowski have recently reconstructed Nazarov's proof, their work will appear elsewere.}
For any locally finite~$\Co^l$-valued charge~$\mu$ on~$\R^d$ and any~$\eps > 0$ there exists~$x\in\supp\mu$\textup,~$a\in \Co^l \setminus \{0\}$\textup, and a non-zero measure~$\nu$ such that
\begin{enumerate}[1\textup{)}]
\item $a\otimes \nu \in \Tan[\mu,x]$\textup,
\item $\ldH \nu(0) \leq \ldH \mu + \eps$\textup,
\item $\nu \in \Sw'(\R^d)$\textup,
\item furthermore\textup, if~$\mu \in \Sw'(\R^d,\Co^l)$\textup, then $a\otimes \nu$ is tangent to~$\mu$ in the Schwartz sense\textup, i.e. 
\eq{\label{SchwartzTangent}
c_jT_{x,r_j}[\mu] \stackrel{\Sw'(\R^d,\Co^l)}{\longrightarrow}a\otimes \nu
}
for some sequences~$c_j > 0$ and~$r_j \searrow 0$.
\end{enumerate}
\end{Th}
The reader may find statements in the spirit of Theorem~\ref{TangentMeasureTheorem} in~\cite{GuidoIsola2006}. The second step in the proof of Theorem~\ref{DimensionalEstimate} is somehow similar to the approach of~\cite{RoginskayaWojciechowski2006}, however, we will phrase it in terms of the simplest possible sharp Harnack's inequality (see Lemma~\ref{BCTArbitraryKernel} below). See in~\cite{Stolyarov2020HLS} how similar sharp Harnack's inequalities for the heat equation lead to sharp versions of the Sobolev embedding theorems for the spaces~$\BV^\Omega$ (so-called Bourgain--Brezis inequalities). We refer the reader to the papers,~\cite{BourgainBrezis2007}~\cite{BBM2004},~\cite{BousquetVanSchaftingen2014},~\cite{HernandezSpector2020},~\cite{LanzaniStein2005},~\cite{Mazya2010},~\cite{Raita2019}, and~\cite{VanSchaftingen2013} for Bourgain--Brezis inequalities; the list is far from being complete, see the survey~\cite{VanSchaftingen2014} as well. The author believes that sharp Harnack inequalities is a common reason both for many Bourgain--Brezis inequalities and dimensional estimates for Fourier constrained measures. The paper~\cite{ASW2018} confirms this belief in the discrete setting. The link between Sobolev embeddings and dimensional estimates had been emphasized already in~\cite{StolyarovWojciechowski2014}. In the classical setting of the first gradient, both follow from isoperimetric inequalities via the co-area formula, see the books~\cite{AFP2000} and~\cite{Mazya2011}.
\begin{Th}\label{AnalyticTheorem}
Let~$\nu$ be a tempered measure on~$\R^d$. Let~$L$ be a~$k$-dimensional plane in~$\R^d$\textup,~$k \in [0..d-1]$. If the set
\eq{\label{ProjectionOfSpectrum}
\set{\zeta \in S^{d-1}}{\exists \lambda > 0 \hbox{ such that }\lambda\zeta \in \spec \nu}
}
does not intersect some neighborhood of~$L\cap S^{d-1}$, then for any~$x\in \R^d$ we have~$\ldH\nu(x) \geq k$.
\end{Th} 
We will also need yet another function space
\eq{\label{DistributionW}
\W = \Set{f\in \Sw'(\R^d,\Co^l)}{ \pi_{\Omega(\xi/|\xi|)^{\perp}}[\hat{f}]\cdot H = 0}.
}
Here~$H$ is an auxiliary non-negative smooth function, positive outside the origin, and having deep zero  at the origin (this means~$H(x) = o(|x|^N)$ for any~$N$ in a neighborhood of the origin). Clearly, the definition does not depend on the particular choice of~$H$. The space~$\W$ is closed in~$\Sw'(\R^d,\Co^l)$, such spaces were firstly introduced in~\cite{AyoushWojciechowski2017}.

\begin{proof}[Derivation of Theorem~\ref{DimensionalEstimate} from Theorems~\ref{TangentMeasureTheorem} and~\ref{AnalyticTheorem}.]
Assume the contrary, let there exist a non-zero charge~$\mu \in \BV^\Omega$ such that~$\ldH\mu < k$. By Theorem~\ref{TangentMeasureTheorem}, there exists a measure~$\nu \in \Sw'(\R^d)$ such that~$a\otimes \nu \in \W$ for some non-zero~$a\in \Co^l$ (we use that~$\W$ is a closed dilation and translation invariant subspace of~$\Sw'(\R^d,\Co^l)$)  and~$\ldH\nu(0) < k$. By~\eqref{DistributionW}, this means
\eq{
\spec \nu \subset \set{\lambda \zeta}{a\in \Omega(\zeta), \quad \lambda > 0} \cup \{0\}
}
since
\eq{
\pi_{\Omega(\xi/|\xi|)^\perp} [a\otimes \hat{\nu}] = \pi_{\Omega(\xi/|\xi|)^\perp}[a]\otimes \hat{\nu}.
}
Since~$\nu$ is a measure (i.e. a non-negative real-valued distribution), its spectrum is symmetric. Thus,
\eq{
\set{\zeta \in S^{d-1}}{\exists \lambda > 0 \hbox{ such that }\lambda\zeta \in \spec \nu} \subset \Omega^{-1}(a)\cap (-\Omega^{-1}(a)).
}
By~\eqref{Antisymmetry}, there exists a~$k$-dimensional subspace~$L_a$ that does not intersect the closed set~$\Omega^{-1}(a)\cap (-\Omega^{-1}(a))$. In the case~$k=d$, this means~$\spec \nu \subset \{0\}$, i.e.~$\nu$ is a polynomial. This leads to contradiction with~$\ldH\nu(0) < d$.

In the case~$k < d$, a neighborhood of~$L_a\cap S^{d-1}$ does not intersect the set~\eqref{ProjectionOfSpectrum}. So,~$\nu$ falls under the scope of Theorem~\ref{AnalyticTheorem}, which implies~$\ldH\nu(0) \geq k$. We get a contradiction.
\end{proof}
\begin{Rem}
In fact\textup, the assertion of Theorem~\textup{\ref{DimensionalEstimate}} is true for the case where~$\mu$ is a charge of locally bounded variation in~$\W$. Moreover\textup, if the mapping~$\Omega$ is rational as a function of~$\zeta$\textup, then we may apply the same reasoning to the case where~$\mu$ is a charge of locally bounded variation satisfying the corresponding system of differential equations \textup(the assumption that~$\mu$ is tempered is not needed in the rational case\textup). 
\end{Rem}
The proof of Theorem~\ref{TangentMeasureTheorem} is presented in Section~\ref{S2} and the proof of Theorem~\ref{AnalyticTheorem} may be found in Section~\ref{S3} below.

I am grateful to Rami Ayoush, Fedor Nazarov, and Michal Wojciechowski for sharing their ideas with me. I also wish to thank Rami Ayoush and Anton Tselishchev for reading the text and exposition advice.

\section{Tangent measure considerations}\label{S2}
We will use the expression for the lower Hausdorff dimension of a measure in terms of the local lower Hausdorff dimensions (see Proposition~$10.2$ in~\cite{Falconer1997}):
\eq{\label{LocalGlobalDimension}
\dH \mu = \sup\set{s}{\hbox{ for } \mu \hbox{ almost all } x \quad \ldH \mu (x) \geq s}.
}
\begin{Rem}\label{LocDimAltDef}
It is convenient to describe the local dimension in a slightly different way\textup:~$\ldH\mu(x)$ is a number~$\alpha$ such that for any~$\eps > 0$ the inequality
\eq{
|\mu|(B_r(x)) \leq r^{\alpha -\eps}
}
holds true for all sufficiently small radii~$r > 0$\textup, however, there exists a sequence~$\{r_j\}_j$\textup,~$r_j \to 0$\textup, such that
\eq{
|\mu|(B_{r_j}(x)) \geq r_j^{\alpha + \eps}.
}
\end{Rem}
We start with a localized version of Theorem~\ref{TangentMeasureTheorem}.
\begin{St}\label{TangentMeasureLocal}
Let~$\mu$ be a locally finite measure on~$\R^d$\textup, let~$x\in \supp \mu$\textup, let also~$\beta > 0$ be some fixed number\textup, and let~$r_0 > 0$. Assume that for all~$r\leq r_0$ it is true that
\eq{\label{BetaAbove}
\mu(B_r(x)) \leq r^{\beta},
}
however\textup, for some~$\eps > 0$ there exists a sequence~$\{r_j\}_{j}$ tending to zero such that
\eq{
\mu(B_{r_j}(x)) \geq r_j^{\beta + \eps}.
}
If~$\eps$ is sufficiently small \textup(the smallness depends on~$\beta$ only\textup)\textup, there exists a non-zero~$\nu \in \Tan[\mu,x]$ such that
\begin{enumerate}[1\textup{)}]
\item $\ldH\nu(0) \leq \beta + \eps$\textup,
\item $\nu(B_R(0)) \leq R^{\beta+2\eps}$ for~$R > 1$\textup, in particular\textup,~$\nu \in \Sw'(\R^d)$\textup,
\item if~$\mu \in \Sw'(\R^d)$\textup, then~$\nu$ is tangent to~$\mu$ in the Schwartz sense as in~\textup{\eqref{SchwartzTangent}}.
\end{enumerate}
\end{St}
\begin{Rem}
In particular\textup, our assumptions imply~$\ldH\mu (x) \in [\beta,\beta+\eps]$.
\end{Rem}
\begin{proof}[Proof of Theorem~\ref{TangentMeasureTheorem} assuming Proposition~\ref{TangentMeasureLocal}.]
Fix sufficiently small~$\eps > 0$. Let~$\alpha = \ldH\mu(x)$ for brevity. According to formula~\eqref{LocalGlobalDimension}, the set of points~$x$ such that
\eq{\label{DimensionInInterval}
\ldH\mu(x) \in [\alpha,\alpha + \eps],
} 
has positive measure~$|\mu|$.
By the Besicovitch differentiation theorem, for~$|\mu|$-almost every~$x\in\R^d$,~$\mu$ has density (which is the unit vector~$\frac{d\mu}{d|\mu|}\in\Co^l$) with respect to~$|\mu|$. Moreover, for~$|\mu|$ almost all~$x$,
\eq{\label{BesDifImpr}
\lim\limits_{\rho\to 0}\frac{1}{|\mu|(B_\rho(x))}\int\limits_{B_\rho(x)} \Big|\frac{d\mu}{d|\mu|}(y) - \frac{d\mu}{d|\mu|}(x)\Big|\,d|\mu|(y) = 0;
}
see Remark~$2.15 (3)$ in~\cite{Mattila1995}.

Let~$x\in \R^d$ be such that~\eqref{DimensionInInterval} and~\eqref{BesDifImpr} hold true. We apply Proposition~\ref{TangentMeasureLocal} to the measure~$|\mu|$ and the point~$x$ with~$\beta = \ldH\mu(x) \in [\alpha,\alpha + \eps]$ (or slightly smaller than~$\ldH\mu(x)$ to ensure~\eqref{BetaAbove}, see Remark~\ref{LocDimAltDef}). It remains to prove the implication 
\eq{\label{ConvMeasImplConvVar}
\frac{1}{|\mu|(B_{r_j}(x))} T_{x,r_j} [|\mu|] \stackrel{j\to \infty}{\longrightarrow} \nu \quad \Longrightarrow \frac{1}{|\mu|(B_{r_j}(x))} T_{x,r_j} [\mu] \stackrel{j\to \infty}{\longrightarrow} \frac{d\mu}{d|\mu|}(x) \otimes\nu
}
both in the classical sense (weak convergence of measures) and in the Schwartz sense.

Let us first prove~\eqref{ConvMeasImplConvVar} for the classical convergence. Pick an arbitrary continuous compactly supported function~$f$, let~$j$ tend to infinity, and assume the first convergence in~\eqref{ConvMeasImplConvVar} holds true. Then, by the Uniform Boundedness Principle,~$|\mu|(B_{Dr_j}(x)) \lesssim |\mu|(B_{r_j}(x))$ for any fixed~$D > 0$. In view of this,
\mlt{
\frac{1}{|\mu|(B_{r_j}(x))}\int\limits_{\R^d} f(z) dT_{x,r_j} [\mu](z) = \frac{1}{|\mu|(B_{r_j}(x))}\int\limits_{\R^d} f(r_j^{-1}(y-x))\,d\mu(y) =\\ \frac{1}{|\mu|(B_{r_j}(x))}\int\limits_{\R^d} f(r_j^{-1}(y-x))\frac{d\mu}{d|\mu|}(y)\,d|\mu|(y)\Eeqref{BesDifImpr} \frac{d\mu}{d|\mu|}(x)\frac{1}{|\mu|(B_{r_j}(x))}\int\limits_{\R^d} f(r_j^{-1}(y-x))\,d|\mu|(y) + o(1)=\\
\frac{d\mu}{d|\mu|}(x)\frac{1}{|\mu|(B_{r_j}(x))}\int\limits_{\R^d} f(z)\,dT_{x,r_{j}}[|\mu|](z) + o(1).
}
Therefore, the second convergence in~\eqref{ConvMeasImplConvVar} holds true. For the Schwartz case, we only need to recall that a sequence of measures~$\{\mathrm{m}_n\}$ converges in the Schwartz sense if and only if it converges in the classical sense, and is uniformly tempered (i.e. there is a uniform bound of the type~$|\mathrm{m}_n|(B_R) \lesssim R^N$,~$R > 1$, with~$N$ and the multiplicative constant independent of the sequence index), and observe that~$|\mu|$ dominates~$\mu$ in this sense.
\end{proof}
Consider two functions on the half line:
\alg{
f(t) = -\log \mu(B_{e^{-t}}(x));\\
h(t) = f(t) - (\beta + 2\eps) t,
}
here~$t \geq 0$ (without loss of generality, we may assume~$r_0 =1$ in Proposition~\ref{TangentMeasureLocal}).
The proof of Proposition~\ref{TangentMeasureLocal} is, in fact, a tedious analysis of these two functions of a single variable.  We encourage the reader to imagine the graphs of these functions, which makes the things more clear than the bulky formulas we present below. The assumptions of Proposition~\ref{TangentMeasureLocal} are restated as
\eq{\label{below}
f(t) \geq \beta t;\quad h(t) \geq -2\eps t,\qquad \hbox{for any}\ t > 0,
}
and there exists a sequence~$\{t_j\}_j$, $t_j \to \infty$, such that
\eq{\label{above}
f(t_j) \leq (\beta + \eps) t_j;\quad h(t_j) \leq -\eps t_j.
}
\begin{Rem}
The function~$f$ is non-decreasing and continuous from the left since we consider closed balls in its definition. Consequently\textup, the function~$h$ is also continuous from the left and satisfies the local Lipschitz bound
\eq{\label{LowerLipschitz}
\frac{h(t) - h(s)}{t-s} \geq -(\beta + 2\eps),\quad t \geq s.
}
\end{Rem}
Consider the function~$h^-$ given by the formula
\eq{
h^-(t) = \inf\limits_{0\leq s\leq t}h(s).
}
This function does not increase and satisfies the inequality~$h^-(t) \geq -2\eps t$ by~\eqref{below}. Clearly,~$h^- \leq h$.
\begin{Le}
The function~$h^-$ is continuous.
\end{Le}
\begin{proof}
Since~$h^-$ does not increase, the desired continuity will follow from the inequality
\eq{\label{Eq216}
\frac{h^-(t) - h^-(s)}{t-s} \geq -(\beta + 2\eps),\quad t \geq s.
}
To show~\eqref{Eq216}, assume~$h^-(t) < h^-(s)$, otherwise there is nothing to prove. Under this assumption, we may prove a stronger bound
\eq{
\frac{h^{-}(t) - h(s)}{t-s} \geq -(\beta + 2\eps),\quad t \geq s.
}
The latter inequality follows from the definition of~$h^-$, our assumption~$h^-(t) < h^-(s)$, and~\eqref{LowerLipschitz}.
\end{proof}
We treat~$f$ and~$h$ as functions of time, thus calling the points on the half line moments.
\begin{Def}
A real~$t\in \R_+$ is called a constancy moment for~$h^-$ provided the latter function is constant on the interval~$(t_-,t_+)$ for some~$t_- < t < t_+$. The set of all constancy moments is called the constancy set. The remaining moments are decrement moments\textup, and the corresponding set is the decrement set.
\end{Def}
\begin{Rem}\label{IntervalsInConstancySet}
The constancy set is open and\textup, thus\textup, it is a union of finite or countable collection of disjoint open intervals.
\end{Rem}
\begin{Le}
If~$t$ is  a decrement moment\textup, then~$h^-(t) = h(t)$.
\end{Le}
\begin{proof}
Assume the contrary, let~$t$ be a decrement moment such that~$h^-(t) < h(t)$. Since~$h$ is continuous from the left,~$h^-(t) < h(\theta)$ for all~$\theta$ in a left neighborhood of~$t$. Therefore,~$h^-$ is constant on the left of~$t$. Since we have assumed~$t$ is not a constancy moment,~$h^-(\theta) < h^-(t)$ for all~$\theta$ in a right neighborhood of~$t$. Therefore, there exists a sequence~$\{t_j\}_j$ tending to~$t$ from the right such that~$h(t_j) < h^-(t) < h(t)$, which contradicts~\eqref{LowerLipschitz}.
\end{proof}

At this point we are ready to make a simple but curious observation. In particular, Lemma~\ref{TangentMeasureExists} below implies that the tangent cone is non-empty for almost all points (the assumptions of the lemma below are fulfilled at almost every point if we set~$\beta \geq d$). See~\cite{Preiss1987} for the classical proof of that fact.
\begin{Le}\label{TangentMeasureExists}
Let~$\mu$ be a locally bounded measure in~$\R^d$\textup, let~$x\in \supp \mu$\textup, let~$\beta \geq 0$ be a fixed number. If there exists a sequence~$\{r_j\}_j$ tending to zero such that~$\mu(B_{r_j}(x)) \geq r_j^{\beta+\eps}$\textup, then~$\Tan[\mu,x] \ne \varnothing$.
\end{Le}
\begin{proof}
Consider the functions~$f$,~$h$, and~$h^-$ contructed from~$\mu$ in the way described above. Note that in previous reasonings we have not used the assumption~\eqref{below} yet, so it is legal to use the results of these considerations in the proof. The assumption~\eqref{above} (which is present in the statement of the lemma), in particular, implies that the corresponding function~$h^-$ has arbitrarily large decrement moments. Let~$\{\tau_j\}_j$ be a sequence of decrement moments tending to infinity. We will show that the sequence
\eq{\label{SequenceOfMeasures}
\Big\{\frac{1}{\mu(B_{e^{-\tau_j}}(x))}T_{x,e^{-\tau_j}}[\mu]\Big\}_j
}
is pre-compact in the space of measures (in the classical sense). It suffices to establish the bounds
\eq{\label{QuantificationOfConv}
\frac{T_{x,e^{-\tau_j}}[\mu](B_{\rho}(0))}{\mu(B_{e^{-\tau_j}}(x))} \leq \rho^{\beta + 2\eps}, \quad \rho > 1,
}
provided~$j$ is sufficiently large (depending on~$\rho$). This inequality is equivalent to
\eq{
\mu(B_{e^{-\tau_j + \log\rho}}(x)) \leq \rho^{\beta + 2\eps}\mu(B_{e^{-\tau_j}}(x)),
}
which may be rewritten as
\eq{
-\log \mu (B_{e^{-\tau_j + \log\rho}(x)}) \geq -(\beta + 2\eps)\log \rho - \log \mu (B_{e^{-\tau_j}}(x)).
}
In terms of the function~$h$, this means
\eq{
h(\tau_j - \log \rho) \geq h(\tau_j), \quad \tau_j \geq \log \rho,
}
which is true:
\eq{
h(\tau_j -\log \rho) \geq h^-(\tau_j-\log \rho) \Gref{\scriptscriptstyle \rho > 1} h^{-}(\tau_j) = h(\tau_j)
}
since~$\tau_j$ is a decrement moment and~$\tau_j \geq \log \rho$.
\end{proof}
\begin{Rem}\label{DecrementRemark}
We have proved a stronger statement\textup: if~$\{\tau_j\}_j$ is a sequence of decrement moments\textup, then the sequence of measures~\textup{\eqref{SequenceOfMeasures}} is pre-compact. Any non-zero tangent measure~$\mathrm{m}$ constructed this way satisfies
\eq{
\mathrm{m} (B_R(0)) \leq R^{\beta + 2\eps},\quad R \geq 1.
}
In particular\textup,~$\mathrm{m} \in \Sw'(\R^d)$. Moreover\textup, if~$\mu$ is tempered\textup, then~$\mathrm{m}$ is tangent to~$\mu$ in the Schwartz sense since the sequence~\textup{\eqref{SequenceOfMeasures}} is uniformly tempered in this case. Indeed\textup, if~$\mu$ satisfies the bound
\eq{
\mu(B_R(0)) \lesssim R^N \mu(B_1(0))
}
for some~$N > \beta + 2\eps$\textup, then we may combine it with~\textup{\eqref{QuantificationOfConv}} for~$\tau_j = \log \rho$ to get
\eq{
\frac{T_{x,e^{-\tau_j}}[\mu](B_{R}(0))}{\mu(B_{e^{-\tau_j}}(x))} \lesssim e^{(\beta+2\eps)\tau_j}\Big(Re^{-\tau_j}\Big)^N \leq R^N,\quad R > e^{\tau_j}.
}
\end{Rem}

\begin{proof}[Proof of Proposition~\ref{TangentMeasureLocal}]
Let~$\{t_j\}_j$ be the sequence of moments from~\eqref{above}. Without loss of generality, we may assume~$t_{j+1} \geq 10 t_j$. Let us first prove that the decrement set is large in the sense that
\eq{\label{DecrementLarge}
\Big|\set{t\in [t_j,t_{j+1}]}{t\ \hbox{is a decrement moment}}\Big|\geq \frac{\eps t_{j+1}}{4(\beta + 2\eps)},
}
here the absolute value of a set denotes its Lebesgue measure. Let~$\{I_k\}_k$ be a finite covering by disjoint intervals of the part of the decrement set lying in~$[t_j,t_{j+1}]$. Let~$I_k = [a_k,b_k]$. In such a case, by~\eqref{below} and~\eqref{above},
\eq{
\sum\limits_k (h^-(a_k) - h^-(b_k)) \geq \eps t_{j+1} - 2\eps t_j
}
since~$h^-$ is piecewise constant outside~$\cup I_k$. Therefore, by~\eqref{LowerLipschitz},
\eq{
\sum\limits_k |a_k - b_k| \geq \frac{\eps t_{j+1} - 2\eps t_j}{\beta + 2\eps} \geq \frac{\eps t_{j+1}}{4(\beta + 2\eps)}
}
since we have assumed~$t_{j+1} \geq 10t_j$. Inequality~\eqref{DecrementLarge} is proved.

Consider now an auxiliary sequence~$\{c_m\}_m$, which is rapidly increasing (at least, we assume~$c_{m+1} \geq 2c_m$). The precise requirements will be given later. A moment~$t\in [t_j,t_{j+1}]$ is called~$m$-bad  provided~$h(s) \geq h(t)$ for any~$s\in [t+c_m,t + c_{m+1}]$. Let us also define yet another auxiliary non-decreasing sequence~$\{N_j\}_j$ of integers that tends to infinity very slowly (in fact, it will be stable most of the time), more specifically, we require that
\eq{\label{cSmallert}
c_{N_{j}+1} \leq \frac{t_j}{100}.
}
In particular, the choice of the~$N_j$ depends on the choice of the~$c_m$. Let us show that there are little~$m$-bad decrement moments with~$m \leq N_{j}$ in the interval~$[t_j,t_{j+1}]$. Namely, we want to prove the inequality
\eq{\label{LittleBadMoments}
\Big|\Set{t\in [t_j,t_{j+1}]}{t\ \hbox{is an $m$-bad decrement moment}}\Big| \leq \frac{3c_m}{c_{m+1}} t_{j+1}
}
under the assumption~$m \leq N_j$.

Here the proof of~\eqref{LittleBadMoments} is. If~$t\in [t_{j},t_{j+1}]$ is an~$m$-bad moment, then~$(t+c_{m},t+c_{m+1})\subset C$, where~$C$ is the set of all constancy moments. By Remark~\ref{IntervalsInConstancySet},~$C$ is a disjoint union of open intervals. Let~$I_1, I_2,\ldots, I_M$ be the intervals in the union that intersect~$[t_j+c_m,t_{j+1}+c_{m+1}]$ and whose lengths are at least~$c_{m+1} - c_m$. Clearly,
\eq{\label{QuantityOfIntervalsEstimate}
M \leq \frac{t_{j+1} - t_j}{c_{m+1}-c_m} + 3  \leq \frac{3t_{j+1}}{c_{m+1}}
}
by~\eqref{cSmallert} and the assumption that the sequences~$\{t_j\}$ and~$\{c_m\}$ are at least as fast as geometric sequences. We know that~$m$-bad \emph{decrement} moments lie inside the intervals of length~$c_m$ lying to the left of the~$I_k$ (because a decrement moment cannot lie inside any~$I_k$). In other words, if~$I_k = [a_k,b_k]$ for~$k=1,2,\ldots, M$, then
\eq{
\Set{t\in [t_j,t_{j+1}]}{t\ \hbox{is an $m$-bad decrement moment}} \subset \bigcup_{k=1}^M\,[a_k-c_m,a_k].
}
Therefore,
\eq{
\Big|\Set{t\in [t_j,t_{j+1}]}{t\ \hbox{is an $m$-bad decrement moment}}\Big| \leq M c_m \Leqref{QuantityOfIntervalsEstimate}\frac{3c_m t_{j+1}}{c_{m+1}}.
}
We have proved~\eqref{LittleBadMoments}. 

If we sum the inequalities~\eqref{LittleBadMoments} over all~$m=1,2,\ldots,N_j$, then we see that the set of all bad decrement moments (up to order~$N_j$) inside~$[t_j,t_{j+1}]$ is small:
\eq{
\Big|\Set{t\in [t_j,t_{j+1}]}{t\ \hbox{is an $m$-bad decrement moment for some~$m\in [1..N_j]$}}\Big| \leq 3t_{j+1}\sum\limits_{m=1}^{\infty}\frac{c_m}{c_{m+1}}.
}
If we set~$c_1 = 1$ and~$c_m = K^{m^2}$, we see, that the sum of the series on the right hand-side does not exceed~$\frac{3t_{j+1}}{K-1}$, which is smaller than the right hand-side of~\eqref{DecrementLarge}, provided~$K$ is sufficiently large. Fix such a~$K$, this defines the~$c_m$. Note that after we have fixed~$\{c_m\}_m$, we are free to choose~$\{N_j\}_j$ to be any slowly non-decreasing integer sequence that fulfills~\eqref{cSmallert}. All in all, we have proved that for any~$j$ there exists a decrement moment~$s_j \in [t_j,t_{j+1}]$ that is~$m$-good (i.e. not~$m$-bad) for any~$m \leq N_j$.

It remains to prove that~$\ldH \nu (0) \leq \beta +2\eps$ for any limit point~$\nu$ of the sequence
\eq{\label{SequenceOfDilations}
\Big\{\frac{1}{\mu(B_{e^{-s_j}}(x))}T_{x,e^{-s_j}}[\mu]\Big\}_j.
}
Such a point exists in view of Remark~\ref{DecrementRemark}, . Without loss of generality, we may assume~$\nu$ is the limit of the sequence~\eqref{SequenceOfDilations}. We will show that for any~$m\in \N$ there exists~$\rho \in [e^{-c_m},e^{-c_{m+1}}]$ such that
\eq{\label{LocalDimensionEstimate}
\nu(B_\rho(0)) \geq \rho^{\beta + 2\eps}.
}
This clearly implies~$\ldH \nu (0) \leq \beta +2\eps$. Fix~$m$. Since~$s_j$ is~$m$-good for any sufficiently large~$j$ (when~$N_j > m$), there exists a moment~$u_{j,m}\in [s_j + c_m,s_j + c_{m+1}]$ such that
\eq{\label{GoodnessSake!}
h(u_{j,m}) \leq h(s_j).
} 
Let~$-\log\rho$ be a limit point of the sequence~$\{u_{j,m} - s_j\}_{j}$. Then,~$-\log \rho \in [c_m,c_{m+1}]$. Let us rewrite~\eqref{GoodnessSake!} using the definition of~$h$:
\eq{
-\log \mu (B_{e^{-u_{j,m}}}(x)) - (\beta + 2\eps)u_{j,m} \leq -\log \mu(B_{e^{-s_j}}(x)) - (\beta + 2\eps)s_j,
}
which is equivalent to
\eq{
(\beta+2\eps)(s_j - u_{j,m}) \leq \log\frac{\mu (B_{e^{-u_{j,m}}}(x))}{\mu(B_{e^{-s_j}}(x))}.
}
Passing to the limit (maybe, we need to enlarge the ball~$B_{\rho}(0)$ slightly), we get~\eqref{LocalDimensionEstimate}.
\end{proof}
\section{Considerations with wavefront theory flavor}\label{S3}
We will start our considerations with a very simple form of Harnack's inequality. Let~$\Phi$ be a continuous non-negative bounded function in~$d$ variables such that 
\eq{\label{DecreasingKernel}
\Phi(sx) \leq \Phi(x),\quad \hbox{for any}\ s \geq 1, x\in \R^d.
}
In other words,~$\Phi$ is radially non-increasing (note that we do not assume the radial symmetry). Let us assume also~$\int_{\R^d}\Phi = 1$ and construct the approximate identity
\eq{
\Phi_t(x) = t^{-d}\Phi(t^{-1}x),\quad t > 0,
}
and the corresponding extension operator:
\eq{
\Ext_\Phi[\mu](x,t) = \mu*\Phi_t(x).
}
This formula makes sense for measures that do not have much mass at infinity. Recall that measures are always non-negative in our considerations. 
\begin{Le}\label{BCTArbitraryKernel}
Let~$\mu$ be a finite measure. Then\textup,
\eq{
\Ext_\Phi[\mu](0,t) \leq t^{-d}\Ext_\Phi[\mu](0,1),\quad t\in (0,1].
}
\end{Le}
\begin{proof}
In fact, there is nothing to prove:
\eq{
\Ext_\Phi[\mu](0,t) = t^{-d}\int\limits_{\R^d}\Phi(-t^{-1}x)\,d\mu(x)\Leqref{DecreasingKernel} t^{-d}\int\limits_{\R^d}\Phi(-x)\,d\mu(x) = t^{-d}\Ext_\Phi[\mu](0,1). 
}
\end{proof}
See~\cite{Stolyarov2020HLS} for less trivial variants of Lemma~\ref{BCTArbitraryKernel}.
\begin{proof}[Proof of Theorem~\ref{AnalyticTheorem}]
Without loss of generality, we may assume 
\eq{
L = \set{\xi\in \R^d}{\xi_{k+1} = \xi_{k+2}=\ldots=\xi_{d} = 0}
}
and that~$\nu$ has spectrum inside the set
\eq{\label{SpectrumRestriction}
\Set{\xi\in \R^d}{\sum\limits_{j=k+1}^d\xi_j^2 \geq \eps \sum\limits_{j=1}^k \xi_j^2},
}
where~$\eps$ is a sufficiently small number. Let~$\Psi\in\Sw(\R^d)$ be a positive function with compact spectrum. Consider the measure~$\tilde{\nu} = \Psi\nu$. This measure decays at infinity:
\eq{\label{DecaysAtInfinity}
\forall N \in \N \quad \tilde{\nu}(\R^d\setminus B_R(0)) \lesssim R^{-N},\quad  R > 1,
} 
and clearly,
\eq{
\ldH \nu(0) = \ldH\tilde{\nu}(0).
}
The cost of passing from~$\nu$ to~$\tilde{\nu}$ is that~$\spec \tilde{\nu} \subset \spec \nu + \spec \Phi$
does not necessarily lie inside the set~\eqref{SpectrumRestriction}. However, it lies inside the set
\eq{\label{NewSpectrumRestriction}
\Set{\xi\in \R^d}{\sum\limits_{j=k+1}^d\xi_j^2 \geq \frac{\eps}{2} \sum\limits_{j=1}^k \xi_j^2} \cup \set{\xi \in \R^d}{|\xi| \leq R},
}
where~$R$ is sufficiently large. Let us rename~$\tilde{\nu}$ back to~$\nu$. From now on~$\nu$ is a measure that satisfies~\eqref{DecaysAtInfinity} and  has spectrum inside the set~\eqref{NewSpectrumRestriction}. Note that~$\hat{\nu}$ is a smooth function. Consider the heat extension of~$\nu$:
\eq{
\Heat[\nu](x,t) = (4\pi t)^{-\frac{d}{2}} \int\limits_{\R^d}e^{-\frac{|x-y|^2}{4t}}\,d\nu(y);\quad \Heat[\nu](\fdot,t) = \Big(e^{-4\pi^2 t |\xi|^2}\hat{\nu}(\xi)\Big)\check{\phantom{I}}.
}
It suffices to prove the estimate
\eq{
\Heat[\nu](0,t) \lesssim t^{\frac{k-d}{2}}
}
since~$\nu(B_r(x)) \lesssim r^{d}\Heat[\nu](x,r^2)$ when~$r \leq 1$. Consider the distribution~$\nu_t$ defined on~$\R^{d-k}$ by the formula
\eq{
\hat{\nu}_t(\xi_{k+1},\xi_{k+2},\ldots,\xi_d) = \int\limits_{\R^{k}} \hat{\nu}(\xi) e^{-4\pi^2 t\sum_{j=1}^k\xi_j^2}\,d\xi_{1}d\xi_{2}\ldots d\xi_k,
}
which is at least well defined by~\eqref{NewSpectrumRestriction} (we integrate over a bounded region). Using standard distribution theory, one may verify~$\nu_t$ is a non-negative distribution. Thus, by Schwartz's theorem, it is a tempered measure. Since~$\hat{\nu}_t$ is continuous at the origin,~$\nu_t$ is a finite measure. It remains to write a chain of inequalities:
\mlt{
\Heat[\nu](0,t) = \int\limits_{\R^{d-k}} \hat{\nu}_t\big(\xi_{k+1},\xi_{k+2},\ldots,\xi_{d}\big)e^{-4\pi^2 t\sum_{k+1}^{d}\xi_j^2}\,d\xi_{k+1}d\xi_{k+2}\ldots d\xi_d = \Heat[\nu_t](0,t) \Lref{\hbox{\tiny Lem. \ref{BCTArbitraryKernel}}}\\ t^{\frac{k-d}{2}}\Heat[\nu_t](0,1) =t^{\frac{k-d}{2}}\int\limits_{\R^d} e^{-4\pi^2 (t\sum_{1}^k\xi_j^2 + \sum_{k+1}^d\xi_j^2)}\hat{\nu}(\xi)\,d\xi \leq t^{\frac{k-d}{2}}\nu(\R^d)\int\limits_{\scriptscriptstyle\eqref{NewSpectrumRestriction}} e^{-4\pi^2 \sum_{k+1}^d\xi_j^2}\,d\xi,
} 
which is finite.

\end{proof}
\bibliography{mybib}{}
\bibliographystyle{amsplain}

St. Petersburg State University, Department of Mathematics and Computer Science;

St. Petersburg Department of Steklov Mathematical Institute;

d.m.stolyarov at spbu dot ru.
\end{document}